\documentclass[dvips]{imsart}

\RequirePackage[OT1]{fontenc}
\RequirePackage{amsthm,amsmath}
\RequirePackage{natbib}
\RequirePackage[colorlinks,citecolor=blue,urlcolor=blue]{hyperref}

\arxiv{arXiv:0000.0000}

\startlocaldefs
\numberwithin{equation}{section}
\theoremstyle{plain}

\endlocaldefs

\usepackage[margin=0.98in]{geometry}

\usepackage{mathrsfs}
\usepackage{amssymb}
\usepackage{epsfig}
\usepackage{url}
\usepackage{amsmath}
\usepackage{natbib}
\usepackage{amsthm}
\usepackage{booktabs}
\usepackage{url}
\usepackage{alltt}
\usepackage{enumerate}
\usepackage{multirow}
\usepackage{bm}
\usepackage{color}
\usepackage{algpseudocode}
\usepackage{algorithm}

\newcommand{\rw}{\mathtt{w}}
\newcommand{\rx}{\mathtt{x}}

\newcommand{\rz}{\mathrm{z}}
\newcommand{\rv}{\mathrm{v}}
\newcommand{\vv}{\mathbf{v}}
\newcommand{\vw}{\mathbf{w}}
\newcommand{\vx}{\mathbf{x}}

\newcommand{\vz}{\mathbf{z}}
\newcommand{\mV}{\mathbf{V}}

\newcommand{\mX}{\mathbf{X}}

\newcommand{\nmathbf}{\bm}

\def\bfX{\nmathbf X}

\def\boldfacefake#1{\kern-4pt
    \hbox{ \mathsurround=0pt
    \hbox to 0.4pt{$#1$\hss}\hbox to 0.4pt{$#1$\hss}\hbox {$#1$}}}

\newcommand{\be}{\begin{eqnarray}}
\newcommand{\ee}{\end{eqnarray}}
\newcommand{\ba}{\begin{eqnarray*}}
\newcommand{\ea}{\end{eqnarray*}}

\newcommand{\reals}{\mbox{\rm I\kern-.20em R}}
\newcommand{\sreals}{\mbox{\small \rm I\kern-.20em R}}

\newtheorem{theorem0}{Theorem}
\newtheorem{lemma0}{Lemma}
\newtheorem{remark0}{Remark}
\newtheorem{fact0}{Fact}
\newtheorem{example0}{Example}
\newtheorem{definition0}{Definition}
\newtheorem{corollary0}{Corollary}
\newtheorem{proposition0}{Proposition}
\newtheorem{algorithmY}{Algorithm}
\newtheorem{conjecture0}{Conjecture}

\newenvironment{theorem}{\begin{theorem0} \mbox{} }{\end{theorem0}}
\newenvironment{lemma}{\begin{lemma0} \mbox{}}{\end{lemma0}}

\newenvironment{definition}{\begin{definition0}
\mbox{}}{\end{definition0}}

\newenvironment{proposition}{\begin{proposition0}\mbox{}
}{\end{proposition0}}

\begin{document}

\begin{frontmatter}
\title{An Information-Theoretic Alternative to the Cronbach's Alpha Coefficient of Item Reliability} 
\runtitle{Information-Theoretic Measure of Reliability}

\begin{aug}
\author{\fnms{Ernest} \snm{Fokou\'e}\thanksref{t1,m1}\ead[label=e1]{epfeqa@rit.edu}},
\author{\fnms{Necla} \snm{G\"und\"uz}\thanksref{m2}\ead[label=e2]{ngunduz@gazi.edu.tr}}

\thankstext{t1}{Corresponding author}
\runauthor{Fokou\'e and G\"und\"uz}

\affiliation{Rochester Institute of Technology\thanksmark{m1} and Gazi University\thanksmark{m2}}

\address{\thanksmark{m1}School of Mathematical Sciences\\
Rochester Institute of Technology, Rochester, New York, USA\\
\printead{e1}}

\address{\thanksmark{m2}Department of Statistics, Faculty of Science\\
Gazi University, Ankara, Turkey\\
\printead{e2}}
\end{aug}

\begin{abstract}
We propose an information-theoretic alternative to the popular
Cronbach alpha coefficient of reliability. Particularly suitable for
contexts in which instruments are scored on a strictly nonnumeric scale,
our proposed index is based on functions of the entropy of the distributions of defined on the sample
space of responses. Our reliability index tracks the Cronbach alpha coefficient uniformly while offering
several other advantages discussed in great details in this paper.
\end{abstract}


\end{frontmatter}

\section{Introduction}
Suppose that we are given a dataset represented by an $n \times p$ matrix $\bfX$ whose $i$th row
$\vx_i^\top \equiv (\rx_{i1}, \rx_{i2},\cdots,\rx_{ip})$
denotes the $p$-tuple of characteristics, with each $\rx_{ij} \in \{1,2,3,4,5\}$ representing  the Likert-type level (order) of preference of respondent $i$
on item $j$. This Likert-type score is obtained by translating/mapping the response levels $\{{\tt Strong \, Disagree}, {\tt Disagree}, {\tt Neutral}, {\tt Agree}, {\tt Strongly \, Agree}\}$ into pseudo-numbers $\{1,2,3,4,5\}$.
\begin{center}
\begin{tabular}{|c|c|c|c|c|}
\toprule
{\tt Strong Disagree} & {\tt Disagree} & {\tt Neutral} & {\tt Agree} & {\tt Strongly Agree} \\ \hline
$\bigcirc$ & $\bigcirc$ & $\bigcirc$ & $\bigcirc$ & $\bigcirc$ \\ \hline
$1$ & $2$ & $3$ & $4$ & $5$ \\
\bottomrule
\end{tabular}
\end{center}
\noindent A usually crucial part in the analysis of questionnaire data is the calculation of Cronbach's alpha coefficient which measures the
internal consistency or reliability/quality of the data. Let $X= (X_1,X_2,\cdots,X_p)^\top$ be a $p$-tuple representing the $p$ items of a questionnaire.
Initially proposed by \cite{Cronbach:1951:1} and later used and re-explained extensively by thousands of researchers and practitioners like
\cite{Bland:1997:1} Cronbach's alpha coefficient is a function of the ratio of the sum of the idiosyncratic item variances over the variance of the sum of the items, and is given by
$$
\alpha = \left(\frac{p}{p-1}\right)\left[1-\frac{\sum_{j=1}^p{\mathbb{V}(X_j)}}{\mathbb{V}\left(\sum_{\ell=1}^p{X_\ell}\right)}\right].
$$

\noindent The coefficient of Cronbach $\alpha$ will be $1$ if the items are all the same and $0$ if none is related to another. Because it is depend on the variance of the sum of a group of independent variables and the sum of their variances. If the variables are positively correlated, the variance of the sum will be increased. If the items making up the score are all identical and so perfectly correlated, all the ${\mathbb{V}(X_j)}$ will be equal and ${\mathbb{V}\left(\sum_{\ell=1}^p{X_\ell}\right)}=p^2 {\mathbb{V}(X_j)}$, so that
$\frac{\sum_{j=1}^p{\mathbb{V}(X_j)}}{\mathbb{V}\left(\sum_{\ell=1}^p{X_\ell}\right)} = \frac{1}{p}$ and $\alpha = 1$.\\

The empirical version of Cronbach's alpha coefficient of internal consistency is given by
$$
\widehat{{\alpha}} =
\left(\frac{p}{p-1}\right)\left[1-\frac{\displaystyle \sum_{j=1}^p{\sum_{i=1}^n{\left(\rx_{ij}-\frac{1}{n}\sum_{i=1}^n{\rx_{ij}}\right)^2}}}
{\displaystyle \sum_{i=1}^n{\left(\sum_{j=1}^p{\rx_{ij}}-\frac{1}{n}\sum_{i=1}^n{\sum_{j=1}^p{\rx_{ij}}}\right)^2}}\right].
$$

\begin{definition}
Let $\mathcal{D}=\{\vx_1,\vx_2,\cdots,\vx_n\}$ be a dataset with
$\vx_i^\top = (\rx_{i1}, \rx_{i2},\cdots,\rx_{ip})$.
An observation vector $\vx_i$ will be called a {\it zero variation} vector if
$\rx_{ij}= {\tt constant},\,\, j=1,\cdots,p$. Respondents with
{\it zero variation} response vectors will be referred to as
{\tt single minded} respondents/evaluators.
\end{definition}
In fact, {\it zero variation} responses essentially reduce a $p$ items  survey to a single item survey.

\begin{theorem}
Let $X= (X_1,X_2,\cdots,X_p)^\top$ be a $p$-tuple representing the $p$ items of a questionnaire.
If $X$ is {\it zero variation}, then the Cronbach's alpha coefficient will be equal to $1$.
\end{theorem}

\begin{proof}
If $X= (X_1,X_2,\cdots,X_p)^\top$ is {\it zero variation}, then $X_j=W$ for $j=1,\cdots,p$,
and $\sum_{j=1}^p{X_j}=pW$. As a result, $\sum_{j=1}^p{\mathbb{V}(X_j)}=p\mathbb{V}(W)$ and
$\mathbb{V}\left(\sum_{j=1}^p{X_j}\right)=\mathbb{V}(pW)=p^2\mathbb{V}(W)$. Therefore,
$$\alpha = 
 \left(\frac{p}{p-1}\right)\left[1-\frac{p\mathbb{V}(W)}{p^2\mathbb{V}(W)}\right]
=  \left(\frac{p}{p-1}\right)\left[1-\frac{1}{p}\right] = 1
$$
\end{proof}

We use a straightforward adaptation of the Cronbach's alpha coefficient to measure {\it respondent reliability}.
\begin{definition}
Let $\mathcal{D}=\{\vx_1,\vx_2,\cdots,\vx_n\}$ be a dataset with
$\vx_i^\top = (\rx_{i1}, \rx_{i2},\cdots,\rx_{ip})$. Let the estimated variance of the $i$th respondent be
${\tilde{S}_i^2} = \sum_{j=1}^p{(\rx_{ij}-\bar{\rx}_i)^2/(p-1)}$.
Let $Z_j = \sum_{i=1}^n{\rx_{ij}}$ represent the sum of the scores given by all the $n$
respondents to item $j$. Our respondent reliability is estimated by
$$
\widehat{\tilde{\alpha}} =
\left(\frac{n}{n-1}\right)\left[1-\frac{\displaystyle \sum_{i=1}^n{\sum_{j=1}^p{\left(\rx_{ij}-\frac{1}{p}\sum_{j=1}^p{\rx_{ij}}\right)^2}}}
{\displaystyle \sum_{j=1}^p{\left(\sum_{i=1}^n{\rx_{ij}}-\frac{1}{p}\sum_{j=1}^p{\sum_{i=1}^n{\rx_{ij}}}\right)^2}}\right]
$$
\end{definition}
Given a data matrix $\bfX$, respondent reliability can be computed in practice by simply taking the
Cronbach's alpha coefficient of $\bfX^\top$, the transpose of the data matrix $\bfX$. Let $m$ be the
number of {\it nonzero variation}. If $m \ll p$ and $m/n$ is very small, then respondent reliability will
be very poor.

Despite its widespread use of Likert-type data since it creation, Cronbach's alpha coefficient is rigorously speaking
not suitable for categorical data for the simple reason that averages on ordinal measurements are often difficult
to interpret at best and misleading at worst. For many years researchers working on the clustering
of Likert-type inappropriately resorted to average-driven methods like kMeans clustering. Fortunately, there has been
a surge of contributions to the clustering of categorical data whereby appropriate methods have been used.
At the heart of the clustering of categorical data is the need to define appropriate measure of similarity.
Recognizing the possibility to preprocess Likert-type questionnaire data into a collection of
estimate probability distributions over the sample spaces of responses, many authors have developed
powerful, scalable and highly techniques for clustering categorical data, most of them
based on information-theoretic \cite{Cover:1991:1}
 concepts like entropy \cite{Huang:1998:1}, \cite{Guha:2000:1}, \cite{Barbara:2002:1}, \cite{San:2004:1}, \cite{Li:2004:1}, \cite{Chen:2005:1}, \cite{Li:2006:1}, \cite{Meila:2007:1}, \cite{Cai:2007:1},
mutual information, variation of information \cite{Meila:2003:1},  along with many other distances and measures on probability distributions like
the Bhattacharya distance \cite{Bhattacharya:1943:1}, \cite{Mak:1996:1}, \cite{Choi:2003:1}, \cite{Goudail:2004:1}, \cite{You:2009:1}, \cite{Reyes:2006:1}, the Kullback-Leibler divergence and the Hellinger distance just to name a few. In this paper, we use information-theoretic tools
and concepts to create several measures of internal consistency of questionnaire data.

\section{Information-Theoretic Measures of Internal Data Consistency}
Let $X_j$ represent one of the questions on the questionnaire, and consider the $n$ responses,
$\{\rx_{1j}, \cdots, \rx_{ij}, \cdots,\rx_{nj}\}$ provided by the $n$ evaluators.
Let $\vv_j = (\rv_{j1}, \cdots,\rv_{jk},\cdots,\rv_{jK})^\top$ denote the vector containing the relative frequencies
of each Likert level for question $j$. With a total of $n$ questionnaires collected, we have
\begin{eqnarray}
\widehat{\rv}_{jk} = \frac{1}{n}\sum_{i=1}^n{I({\rm x}_{ij}=k)}, \qquad k=1,2,\cdots,K \quad \text{and} \quad j=1,2,\cdots,p.
\label{eq:entr:item:0}
\end{eqnarray}
Using \eqref{eq:entr:item:0}, one can then form probabilistic vectors $\widehat{\vv}_j^\top = (\widehat{\rv}_{j1}, \cdots, \widehat{\rv}_{jk}, \cdots,\widehat{\rv}_{jK})$, for $j=1,2,\cdots,p$.  Each vector $\widehat{\vv}_j$ essentially represents an approximate probability distribution on the sample space
made up of the $K$ response levels. Using this probabilistic representation of each question $j$,
we can compare the variability of each item of the questionnaire using the entropy, specifically
\begin{eqnarray}
{H}(\widehat{\vv}_j) = -\sum_{k=1}^K{\widehat{\rv}_{jk}\log_2(\widehat{\rv}_{jk})}
\label{eq:entr:item:1}
\end{eqnarray}

We can imagine a transformation of the $n \times p$ data matrix $\mX$ into a probabilistic $p \times K$ counterpart
$\mV$ where each row represent the approximate probability distribution of the corresponding question (item).
The entropy of each question indicates the variability of the answers given by students on that question.
For a given course and a given instructor, a small value of this entropy would indicate a greater degree of agreement of
his/her student on that item, and therefore suggest a more careful examination of the scores on that item. As far as the relationship between
items is concerned, information theory also provides a wealth of measures. The symmetrized Kullback-Leibler
divergence given by
$$
{\tt KL}_2(\vv_i, \vv_j) = \frac{1}{{2}}\Big\{{\tt KL}(\vv_i, \vv_j) +{\tt KL}(\vv_j, \vv_i)\Big\} = \frac{1}{{2}}{\sum_{k=1}^K\left\{\rv_{ik}\log\left(\frac{\rv_{ik}}{\rv_{jk}}\right)+
\rv_{jk}\log\left(\frac{\rv_{jk}}{\rv_{ik}}\right)\right\}},
$$
where
$$
{\tt KL}(\vv_i, \vv_j) = \sum_{k=1}^K{\rv_{ik}\log\left(\frac{\rv_{ik}}{\rv_{jk}}\right)}
\quad \textrm{and} \quad {\tt KL}(\vv_j, \vv_i) =
\sum_{k=1}^K{\rv_{jk}\log\left(\frac{\rv_{jk}}{\rv_{ik}}\right)},
$$
is usually the default measure used by most authors. The {\it Kullback-Leibler} divergence is closely related the {\it mutual information}
$$
I(\vv_i, \vv_j) = \sum_{k=1}^K\left\{\sum_{l=1}^K\left\{\rv_{ik,jl}\log_2\left(\frac{\rv_{ik,jl}}{\rv_{ik}\rv_{jl}}\right)\right\}\right\},
$$
which has been used extensively in machine learning to define a distance known as the
{\it Variation of Information}, and defined by
$$
{\tt VI}(\vv_i, \vv_j) = H(\vv_i) + H(\vv_j)- 2 I(\vv_i, \vv_j).
$$
Many other non-information-theoretic similarity and variation  measures operating on probabilistic vectors can be used to further investigate several
aspects of the categorical data at hand. One that have been extensively used in the machine learning and data mining
community is the Bhattacharya distance  \cite{Bhattacharya:1943:1} is given by
$$
{\tt BC}(\vv_i, \vv_j) = -\log {\tt F}(\vv_i, \vv_j),
$$
where
$$
{\tt F}(\vv_i, \vv_j) = \sum_{k=1}^K{\sqrt{\rv_{ik}\rv_{jk}}},
$$
is known as the Bhattacharya coefficient  or Fidelity coefficient. The Bhattacharya distance
${\tt BC}(\vv_i, \vv_j)$ measures the overlap between $\vv_i$ and $\vv_j$. The Bhattacharya distance
has been immensely used in various data mining and machine learning applications \cite{Mak:1996:1}, \cite{Choi:2003:1}, \cite{Goudail:2004:1},
 \cite{You:2009:1}. It is interesting to note that
the Bhattachrya distance is related to {\it total variation} measure defined by
$$
\Delta(\vv_i, \vv_j) = \frac{1}{{2}}\sum_{k=1}^K{|\rv_{ik}-\rv_{jk}|} = \frac{1}{{2}}\|\vv_i-\vv_j\|_1
$$
where $\|\cdot\|_1$ is the $\ell_1$ norm. Another very commonly used distance is the Hellinger distance between
$\vv_i$ and $\vv_j$ is given by
$$
{\tt Hellinger}(\vv_i, \vv_j) = \frac{1}{\sqrt{2}}\sqrt{\sum_{k=1}^K{(\sqrt{\rv_{ik}}-\sqrt{\rv_{jk}})^2}}
=\frac{1}{\sqrt{2}}\|\sqrt{\vv_i}-\sqrt{\vv_j}\|_2,
$$
where $\|\cdot\|_2$ is the Euclidean norm or $\ell_2$ norm,
 $\sqrt{\vv_i} = (\sqrt{\rv_{i1}},\cdots,\sqrt{\rv_{iK}})$ and $\sqrt{\vv_j} = (\sqrt{\rv_{j1}},\cdots,\sqrt{\rv_{jK}})$.

\begin{definition} Let $Q$ denote an instrument (questionnaire) for which the realized matrix of obtained
responses is given by $\bfX$ with entries $\rx_{ij} \in \{1,2, \cdots, K\}$. We propose an information-theoretic
measure of the reliability of $Q$, referred to as the information consistency ratio of $Q$ and given by
\begin{eqnarray}
\varphi = 1 - \frac{\underset{i=1,\cdots,n}{\mathtt{min}}\Big\{H\left(\widehat{\vz}_i\right)\Big\}}
{\underset{\vz}{\mathtt{max}}\Big\{H\left({\vz}\right)\Big\}}
= 1 - \frac{\underset{i=1,\cdots,n}{\mathtt{min}}\Big\{H\left(\widehat{\vz}_i\right)\Big\}}{H\left(\frac{1}{K},\cdots,\frac{1}{K}\right)}= 1 - \frac{\underset{i=1,\cdots,n}{\mathtt{min}}\Big\{H\left(\widehat{\vz}_i\right)\Big\}}{\log_2(K)},
\label{eq:entr:index:1}
\end{eqnarray}
where each $\widehat{\vz}_i = \{\widehat{\rz}_{ik}, \, k=1,2,\cdots,K\}$ defines an approximate
probability distribution on the sample space of possible responses, and $H(\cdot)$ is the entropy function,
with
\begin{eqnarray}
\widehat{\rz}_{ik} = \frac{1}{p}\sum_{j=1}^p{I({\rm x}_{ij}=k)} \quad \text{and} \quad {H}(\widehat{\vz}_i) = -\sum_{k=1}^K{\widehat{\rz}_{ik}\log_2(\widehat{\rz}_{ik})}.
\label{eq:entr:stud:1}
\end{eqnarray}
\end{definition}

\begin{lemma}
Let $\vz$ denote any probability measure defined on some $K$-dimensional sample space, with
each $\rz_k = \Pr\{E_k\}, \, k=1.2,\cdots,K$. Let $H(\cdot)$ denote the entropy function, such
that for every $\vz$, we have $H(\vz) = -\sum_{k=1}^K{{\rz}_{k}\log_2({\rz}_{k})}$.
Then
$$
\underset{\vz}{\mathtt{max}}\Big\{H\left({\vz}\right)\Big\} = \log_2(K).
$$
\end{lemma}

\begin{proof} Since entropy essentially measures uncertainty (disturbance),
the probability measure for which the uncertainty is the largest is the probability measure $\vz^*$ in which all
the events are equally likely, i.e., $\rz_k^* = \Pr\{E_k\}=\frac{1}{K}, \, k=1.2,\cdots,K$.
\begin{eqnarray*}
\underset{\vz}{\mathtt{max}}\Big\{H\left({\vz}\right)\Big\} = H(\vz^*)
= H\left(\frac{1}{K},\cdots,\frac{1}{K}\right) = -\sum_{k=1}^K{\frac{1}{K}\log_2\left(\frac{1}{K}\right)}=\log_2(K).
\end{eqnarray*}
\end{proof}

\begin{proposition}
Let $Q_0$ denote a special questionnaire whose items are all mutually independent (unrelated).
Then the corresponding information consistency ratio $\varphi_0$ of $Q_0$, is such that
$$
\underset{p \rightarrow \infty}{\lim}{\varphi_0} = 0.
$$
\end{proposition}

\begin{proof}
With $Q_0$ denoting a  questionnaire whose items that are all mutually independent (unrelated), the matrix
of realized responses has entries $\rx_{ij}$ that a realization of the discrete uniform distribution
on $\{1,2,\cdots,K\}$, or specifically, $\rx_{ij} \sim {\tt uniform}(1,2,\cdots,K)$. It follows that
for each $i=1,2,\cdots,n$, we must have
$$
\underset{p \rightarrow \infty}{\lim}{\widehat{\rz}_{ik}} =
\underset{p \rightarrow \infty}{\lim}\left\{\frac{1}{p}\sum_{j=1}^p{I({\rm x}_{ij}=k)}\right\} = \frac{1}{K}, \quad k=1,2,\cdots,K.
$$
In other words, given enough questions (items), the empirical proportion of answers will converge to its
theoretical counterpart by the law of large number. We therefore have the uniform generation of answers, the
limiting distribution
$$
\underset{p \rightarrow \infty}{\lim}{\widehat{\vz}_{i}} = \vz^* =  \left(\frac{1}{K},\cdots,\frac{1}{K}\right).
$$
Finally, since all the response distributions will tend to converge to the same maximal
measure $\vz^*$, i.e. $\widehat{\vz}_{i} \overset{\mathscr{D}}{\rightarrow} \vz^*$, for $i=1,2,\cdots,n$, we must have
$$
\underset{i=1,\cdots,n}{\mathtt{min}}\Big\{H\left(\widehat{\vz}_i\right)\Big\} \overset{P}{\rightarrow} H(\vz^*) = \underset{\vz}{\mathtt{max}}\Big\{H\left({\vz}\right)\Big\},
$$
and therefore
$$
\underset{p \rightarrow \infty}{\lim}{\varphi_0} = 1 - \frac{\underset{i=1,\cdots,n}{\mathtt{min}}\Big\{H\left(\widehat{\vz}_i\right)\Big\}}
{\underset{\vz}{\mathtt{max}}\Big\{H\left({\vz}\right)\Big\}} = 1 - \frac{H(\vz^*)}{H(\vz^*)} = 1-1 = 0.
$$
\end{proof}

\begin{proposition}
Let $Q_+$ denote a special questionnaire whose items are all identical.
Then the corresponding information consistency ratio $\varphi_+$ of $Q_+$, is such that ${\lim}{\varphi_+} = 1$.
\end{proposition}

\begin{proof}
With $Q_+$ denoting a  questionnaire whose items that are all identical, the matrix
of realized responses has entries $\rx_{ij} = c$, for some constant
$c \in \{1,2,\cdots,K\}$. Then for each $i=1,2,\cdots,n$, there exists
$k_+ \in \{1,2,\cdots,K\}$ such that
$$
\widehat{\rz}_{ik} = \left\{
\begin{array}{ll}
1 & \quad k = k_+ \\
0 & \quad k \neq k_+
\end{array}\right.
$$
In other words, with $Q_+$, the approximate distributions $\widehat{\vz}_i$ of the answers of each respondent
are of the form $(1,0,\cdots,0)$, or $(0,1,\cdots,0)$ or $(0,0,\cdots,1)$. Therefore, for $Q_+$, we must have
$H(\widehat{\vz}_i) = 0,  \quad i=1,\cdots,n$, with the result being $\underset{i=1,\cdots,n}{\mathtt{min}}\Big\{H\left(\widehat{\vz}_i\right)\Big\}=0$, and therefore
$$
\varphi_+ = 1 - \frac{\underset{i=1,\cdots,n}{\mathtt{min}}\Big\{H\left(\widehat{\vz}_i\right)\Big\}}
{\underset{\vz}{\mathtt{max}}\Big\{H\left({\vz}\right)\Big\}} = 1 - \frac{0}{H(\vz^*)} = 1-0 = 1.
$$
\end{proof}

\begin{definition}
\label{def:phi:2}
Let $Y_i$ represent the most frequently occurring answer in respondent $i$'s vector of $p$ answers. It is
easy to see that $Y_i$ has the same sample space as each question/item, namely the same Likert scale in our case.
Using the random variables $Y_i$, we can then define $\widehat{\vw} = (\widehat{\rw}_{1}, \cdots,\widehat{\rw}_{k},\cdots,\widehat{\rw}_{K})^\top$ in the same manner that we define
$\vv_j$ earlier. More specifically, we have
\begin{eqnarray}
Y_i = \underset{k=1,\cdots,K}{\mathtt{argmax}}\left\{\frac{1}{p}\sum_{j=1}^p{I(X_{ij}=k)}\right\}
\quad \text{and} \quad
\widehat{\rw}_{k} = \frac{1}{n}\sum_{i=1}^n{I(Y_{i}=k)}.
\label{eq:entr:max:1}
\end{eqnarray}
The entropy of $\widehat{\vw}$ is given by
\begin{eqnarray}
H(\widehat{\vw}) = -\sum_{k=1}^K{\widehat{\rw}_{k}\log_2(\widehat{\rw}_{k})}.
\label{eq:entr:max:2}
\end{eqnarray}
{\it The random variable $Y_i$ is maximal in a set-theoretic sense, and
and can be thought of as the categorical analogue of the sum of numeric $X_j$'s.}
Using $\widehat{\vw}$, an alternative definition of the {\it
information consistency ratio} $\varphi$ is
\begin{eqnarray}
\varphi=
1 - \frac{\underset{i=1,\cdots,n}{\mathtt{min}}\Big\{H\left(\widehat{\vz}_i\right)\Big\}}
{H\left({\widehat{\vw}}\right)}.
\label{eq:entr:index:2}
\end{eqnarray}
An even more stringent measure of the information consistency ratio is given by
\begin{eqnarray}
\varphi=
1 - \frac{\underset{i=1,\cdots,n}{\mathtt{max}}\Big\{H\left(\widehat{\vz}_i\right)\Big\}}
{H\left({\widehat{\vw}}\right)}.
\label{eq:entr:index:3}
\end{eqnarray}
\end{definition}

\section{Demonstration of Properties of $\varphi$}
We use a simple simulation setup to empirically compare the different measures presented in this paper.
We set $p=50$ and $n=1000$ and we vary the ratio of perfectly reliable components from $10\%$ to
$100\%$ by $10\%$. For $i=1,\cdots,n$ and $j=1,\cdots,n$, draw the $\rx_{ij}$'s uniformly with replacement from
$\{1,2,\cdots,K\}$, that is,
$$
{\tt Draw} \quad \rx_{ij} \sim {\tt uniform}(1,2,\cdots,K).
$$
Randomly replace $100{\tt c}\%$ of the columns of $\mX$ with the same column of constant values, where
${\tt c} \in \{0.1,0.2,\cdots,0.9,1\}$.
Table \eqref{tab:empirical:entr:cronb:1} shows the simulated values of the information consistency ratio
and Cronbach's alpha coefficient for different fractions of of reliable components in the instrument.
Figure \eqref{fig:entropy:vs:cronbach:1} is a direct pictorial representation of the
numbers from Table \eqref{tab:empirical:entr:cronb:1}, and we can see that the Cronbach alpha coefficient
is less strick than the information consistency ratio.
\begin{table}[!htbp]
{\begin{tabular}{@{\extracolsep{5pt}} lrrrrrrrrrr}
\\[-1.8ex]\toprule \\[-1.8ex]
Fraction & $10$ & $20$ & $30$ & $40$ & $50$ & $60$ & $70$ & $80$ & $90$ & $100$ \\ \hline
$\varphi_1$ & $0.230$ & $0.270$ & $0.330$ & $0.440$ & $0.520$ & $0.630$ & $0.740$ & $0.900$ & $1.000$ & $1.000$ \\
$\varphi_2$ & $0.000$ & $0.000$ & $0.020$ & $0.080$ & $0.140$ & $0.240$ & $0.360$ & $0.520$ & $0.720$ & $1.000$ \\
$\varphi_3$ & $0.230$ & $0.270$ & $0.330$ & $0.440$ & $0.520$ & $0.630$ & $0.740$ & $0.900$ & $1.000$ & $1.000$ \\
$\varphi_4$ & $0.000$ & $0.000$ & $0.020$ & $0.080$ & $0.140$ & $0.240$ & $0.360$ & $0.520$ & $0.720$ & $1.000$ \\
${\tt Cronbach}$ & $0.380$ & $0.700$ & $0.820$ & $0.910$ & $0.940$ & $0.960$ & $0.980$ & $0.990$ & $1.000$ & $1.000$ \\
\bottomrule \\[-1.8ex]
\end{tabular}}
  \caption{Simulated values of the information consistency ratio and Cronbach's alpha coefficient for different fractions
  of reliable components in the instrument.}
\label{tab:empirical:entr:cronb:1}
\end{table}

\begin{figure}
  \centering
  \includegraphics[width=10cm, height=8cm]{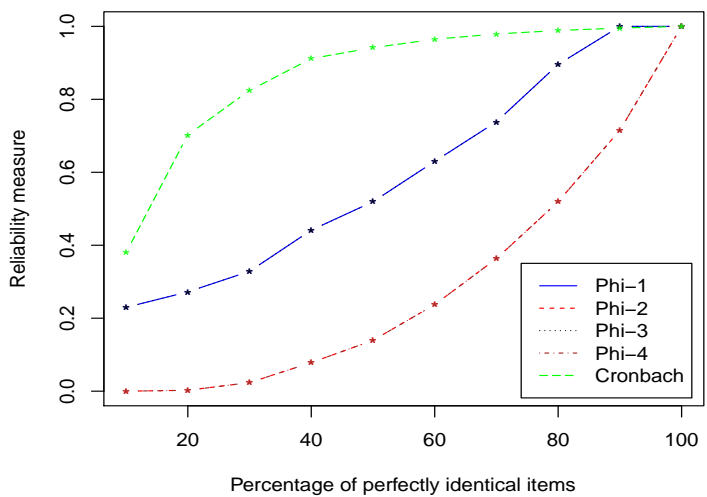}\\
  \caption{Comparative curves of $\varphi$ and Cronbach alpha as measures of internal consistency.}
  \label{fig:entropy:vs:cronbach:1}
\end{figure}

\section{Conclusion and Discussion}
We have proposed and developed an information-theoretic  measure of internal data consistency  et demonstrated
via straightforward simulation that it does indeed capture  the amount of information potentially contained in the data for,the purposes
of performing all kinds of pattern for the data. We have also provided several many other measures of
similarity over probabilistic vectors that we intend to use for further refined our proposed information consistency ratio $\varphi$.
We intend to conduct a larger simulation study to establish our proposed measure on a stronger footing. We also plan to
compare the predictive power of ICR to Cronbach's alpha coefficient on various real and simulated data.

\section*{Acknowledgements}
Ernest Fokou\'e wishes to express his heartfelt gratitude and infinite thanks to Our Lady of Perpetual Help for Her
ever-present support and guidance, especially for the uninterrupted flow of inspiration received through Her
most powerful intercession.

\bibliographystyle{chicago}
\bibliography{fg-entropy-ref}

\end{document}